\documentclass[USenglish,reqno,intlimits]{amsart}
\usepackage{amssymb, amsthm, amscd, mathrsfs}

\usepackage[affil-it]{authblk} 
\usepackage[foot]{amsaddr}

\usepackage{paralist}

\edef\restoreparindent{\parindent=\the\parindent\relax}
\usepackage{parskip}
\restoreparindent

\usepackage{bigints}
\usepackage{empheq}
\usepackage{cases}
\usepackage{enumitem}
\usepackage{dsfont}
\usepackage{cite}

\usepackage{hyperref}
\makeatletter
\def\specialsection{\@startsection{section}{1}%
  \z@{\linespacing\@plus\linespacing}{.5\linespacing}%
  {\raggedright\sffamily\LARGE\bfseries}}% NEW
\def\section{\@startsection{section}{1}%
  \z@{.9\linespacing\@plus\linespacing}{.7\linespacing}%
  {\raggedright\sffamily\LARGE\bfseries}}% NEW
  \renewcommand{\@secnumfont}{\raggedright\sffamily\Large\bfseries}
\makeatother

\makeatletter
\def\subsection{\@startsection{subsection}{2}%
  \z@{.9\linespacing\@plus.7\linespacing}{.7\linespacing}%
  {\raggedright\sffamily\Large\bfseries}}
\makeatother
\newtheoremstyle{dthm}
  {.5\baselineskip}
  {.5\baselineskip}
  {\itshape}
  {}
  {\sffamily\bfseries}
  {\ifx\thmnote\@gobble.\else\normalfont.\fi}
  {.5em}
  {}
\newtheoremstyle{ddef}
  {.5\baselineskip}
  {.5\baselineskip}
  {\normalfont}
  {}
  {\sffamily\bfseries}
  {\ifx\thmnote\@gobble.\else\normalfont.\fi}
  {.5em}
  {}
\theoremstyle{dthm}
\newtheorem{theorem}{Theorem}
\newtheorem{corollary}{Corollary}
\newtheorem{lemma}{Lemma}
\newtheorem{proposition}{Proposition}
\theoremstyle{ddef}

\newtheorem{remark}{Remark}

\def \dv {\mathrm{div}}
\def \d {\mathrm{d}}

\makeatletter
\def\@settitle{\begin{center}%
  \baselineskip14\p@\relax
    \raggedright\sffamily\Huge\bfseries%<- NEW
\uppercasenonmath\@title
  \@title
  \end{center}%
}
\makeatother
\title[A\MakeLowercase{n inverse problem of radiative potentials and initial temperatures}] %Use the shortened version of the full title
      {A\MakeLowercase{n inverse problem of radiative potentials\\and initial temperatures in parabolic\\equations with dynamic boundary\\conditions}}

\address{\textsf{\textbf{E. M. Ait Ben Hassi, S. E. Chorfi, L. Maniar,} Department of Mathematics, Faculty of Sciences Semlalia, LMDP, UMMISCO (IRD-UPMC), Cadi Ayyad University, Marrakesh 40000, B.P. 2390, Morocco.}}

\email{m.benhassi@gmail.com}	
\email{chorphi@gmail.com}	
\email{maniar@uca.ma}

\patchcmd{\abstract}{3pc}{0pt}{}{}

\begin{document}
\begin{abstract}
\normalsize
We study an inverse problem involving the restoration of two radiative potentials, not necessarily smooth, simultaneously with initial temperatures in parabolic equations with dynamic boundary conditions. We prove a Lipschitz stability estimate for the relevant potentials using a recent Carleman estimate, and a logarithmic stability result for the initial temperatures by a logarithmic convexity method, based on observations in an arbitrary subdomain.


\end{abstract}

\noindent \textsf{\Large E. M. Ait Ben Hassi, S. E. Chorfi, and L. Maniar}
\vspace{-0.1cm}

\maketitle

\section{Introduction and main results}
This paper deals with the analysis of an inverse problem concerning the reconstruction of non-smooth radiative potentials simultaneously with initial temperatures in parabolic equations with dynamic boundary conditions, based on some internal observed data.

Let $\Omega \subset \mathbb{R}^N$ be nonempty bounded connected open set ($N\geq 2$ is an integer), with boundary $\Gamma=\partial\Omega$ of class $C^2$, and $\nu$ be the normal vector field of $\Gamma$ pointing outside of $\Omega$. We fix $T>0$ and a nonempty subdomain $\omega \Subset \Omega$. In the sequel, we denote $\Omega_T =(0,T)\times \Omega, \, \omega_T =(0,T)\times \omega, \, \Gamma_T =(0,T)\times \Gamma$. Consider the following system
\begin{empheq}[left =\left(\mathcal{E}_{p,q,Y_0}\right)\empheqlbrace]{alignat=2}
\begin{aligned}
&\partial_t y = \dv(A(x) \nabla y) + B(x)\cdot \nabla y + p(x)y, &\text{in } \Omega_T , \\
&\partial_t y_{\Gamma}= \dv_{\Gamma} (D(x)\nabla_\Gamma y_{\Gamma}) -\partial_{\nu}^A y + \langle b(x), \nabla_\Gamma y_{\Gamma} \rangle_{\Gamma} + q(x)y_{\Gamma}, &\text{on } \Gamma_T, \\
&y_{\Gamma}(t,x) = y_{|\Gamma}(t,x), &\text{on } \Gamma_T, \\
&(y,y_{\Gamma})\rvert_{t=0}=Y_0:=(y_0, y_{0,\Gamma}),  &\Omega\times\Gamma, \nonumber
\end{aligned}
\end{empheq}
where $(y_0, y_{0,\Gamma})\in L^2(\Omega)\times L^2(\Gamma)$ are the initial conditions, and all the coefficients in system $\left(\mathcal{E}_{p,q,Y_0}\right)$ are assumed to be bounded,
\begin{equation}
B \in L^\infty(\Omega)^N, \quad p \in L^\infty(\Omega), \quad b\in L^\infty(\Gamma)^N, \; q \in L^\infty(\Gamma). \label{bdcoeff}
\end{equation}
We assume that the diffusion matrices $A$ and $D$ are symmetric and uniformly elliptic, i.e.,
\begin{align}
A=(a_{ij})_{i,j} \in C^1\left(\overline{\Omega}; \mathbb{R}^{N\times N}\right), \quad a_{ij}=a_{ji}, \quad 1\leq i,j\leq N, \label{symA}\\
D=(d_{ij})_{i,j} \in C^1\left(\Gamma; \mathbb{R}^{N\times N}\right), \quad d_{ij}=d_{ji}, \quad 1\leq i,j\leq N, \label{symD}
\end{align}
and there exists a constant $\beta >0$ such that
\begin{align}
A(x)\zeta \cdot \zeta &\geq \beta |\zeta|^2, && x\in \overline{\Omega}, \;\zeta \in \mathbb{R}^N, \label{uellipA}\\
\langle D(x)\zeta, \zeta \rangle_{\Gamma} &\geq \beta |\zeta|_\Gamma^2, && x\in \Gamma, \;\zeta \in \mathbb{R}^N. \label{uellipD}
\end{align}
Here, $``\cdot"$ denotes the standard inner product on $\mathbb{R}^N$ and $\langle \cdot, \cdot\rangle_\Gamma$ is the Riemannian inner product on $\Gamma$ (see definition below). By $y_{|\Gamma}$, one denotes the trace of $y$, while $\partial_{\nu}^A y :=(A\nabla y\cdot \nu)_{|\Gamma}$ is the conormal derivative and $\partial_\nu y:=(\nabla y\cdot \nu)_{|\Gamma}$ is the normal derivative on $\Gamma$. The operator $\dv$ stands for the usual divergence operator carried out with respect to the $x$-variable in $\Omega$. If $\mathrm{g}$ denotes the natural Riemannian metric on $\Gamma$ induced by the inclusion $\Gamma \hookrightarrow \mathbb{R}^N$, the tangential gradient $\nabla_\Gamma y$ (with respect to $\mathrm{g}$) is the part of the standard gradient $\nabla y$ tangent to $\Gamma$, namely, $\nabla_\Gamma y = \nabla y -(\partial_\nu y) \nu$. The divergence operator $\dv_\Gamma$ associated with the Riemannian metric $\mathrm{g}$ is defined locally by $\displaystyle \dv_\Gamma(X)=\frac{1}{\sqrt{|\mathrm{g}|}} \sum\limits_{j=1}^{N-1} \dfrac{\partial}{\partial x^j} \left(\sqrt{|\mathrm{g}|}\, X^j\right)$, where $\displaystyle X = \sum\limits_{j=1}^{N-1} \dfrac{\partial}{\partial x^j} X^j$ is any smooth vector field and we denoted by $\mathrm{g}=\left(\mathrm{g}_{ij}\right)$ the corresponding metric tensor and $|\mathrm{g}|=\det\left(\mathrm{g}_{ij}\right)$. The term $\dv_\Gamma\left(D(\cdot)\nabla y_\Gamma\right)$ represents surface diffusion effects on the boundary $\Gamma$ given by the vector field $D(\cdot)\nabla y_\Gamma: \Gamma \rightarrow T\Gamma$. For $x\in \Gamma$, the inner product and the norm on the tangent space $T_x \Gamma$ are defined by
$$\langle X_1, X_2 \rangle_\Gamma =\sum_{i,j=1}^{N-1} \mathrm{g}_{ij}(x) X_1^i X_2^j, \qquad |X|_\Gamma=\langle X,X\rangle_\Gamma^{1/2}.$$
Since $\Gamma$ is a compact Riemannian manifold without boundary, the following divergence formula holds
\begin{equation}
\int_\Gamma (\dv_\Gamma X)z \,\d S =- \int_\Gamma \langle X, \nabla_\Gamma z \rangle_\Gamma \,\d S, \qquad z\in H^1(\Gamma), \label{sdt}
\end{equation}
where $X$ is any $C^1$ vector field and $\d S$ is the surface measure on $\Gamma$.

In the last recent years, an increasing interest has been devoted to parabolic problems with dynamic boundary conditions, see, e.g., \cite{FGGR'02, Go'06, KMMR'19, Mo'10}. We refer to \cite{Go'06} for the physical interpretation and the derivation of such type of boundary conditions. Recently, Maniar et al. have proven new Carleman estimates for such problems within the scope of null controllability and inverse problems, see \cite{ACMO'19, BCMO'20, KM'19, MMS'17}.
\smallskip

\noindent\textbf{Inverse potentials problem.}\\
Let $T>0$, $\theta \in (0,T)$ and $0<t_0 <t_1 \leq T$ such that $\displaystyle \theta=\frac{t_0 +t_1}{2}$. Denote
$$\mathbb{L}^2=L^2(\Omega)\times L^2(\Gamma), \qquad \mathbb{L}^\infty=L^\infty(\Omega)\times L^\infty(\Gamma)$$
and $$\mathbb{W}^{2,\infty}=\{(y,y_\Gamma) \in W^{2,\infty}(\Omega)\times W^{2,\infty}(\Gamma) \colon y_{|\Gamma} = y_\Gamma\},$$
with corresponding norm $\|\cdot\|_{2,\infty}$. For fixed constants $r>0$ and $R>0$, we denote the set of admissible initial data by
\begin{align}
\mathcal{I}:=\{Y_0=(y_0, y_{0,\Gamma}) \in \mathbb{W}^{2,\infty} \colon y_0, y_{0,\Gamma}\geq r, \|Y_0\|_{2,\infty} \leq R\} \label{init11}
\end{align}
and the set of admissible potentials by
\begin{align}
\mathcal{P}:=\{(p,q) \in \mathbb{L}^\infty \colon \|p\|_\infty ,\|q\|_\infty \leq R\}.
\end{align}
Our purpose is to determine the radiative potentials $p\in L^\infty(\Omega)$ and $q\in L^\infty(\Gamma)$ in $\left(\mathcal{E}_{p,q,Y_0}\right)$ belonging to $\mathcal{P}$, from a single measurement $Y(\theta, \cdot)=(y,y_{\Gamma})\rvert_{t=\theta}$ and extra partial observation on the first component of the solution $y\rvert_{(t_0, t_1)\times\omega}$. These potentials model radiative loss in the heat process, due to low temperature excess either in the domain or on the boundary.

We mainly aim to establish the following Lipschitz stability estimate.
\begin{theorem}\label{thm1}
Let $\omega_{t_0,t_1}=(t_0,t_1)\times \omega$. Consider $Y$ and $\widetilde{Y}$ the corresponding solutions of $\left(\mathcal{E}_{p,q,Y_0}\right)$ and $\left(\mathcal{E}_{\widetilde{p},\widetilde{q},\widetilde{Y}_0}\right)$ respectively, with given $r, R>0$. Then, there exists a positive constant $C=C\left(\Omega, \omega,t_0, t_1, r, R\right)$ such that
\begin{equation}
\|(p-\widetilde{p}, q-\widetilde{q})\|_{\mathbb{L}^{2}} \leq C\left(\|Y(\theta, \cdot)-\widetilde{Y}(\theta, \cdot)\|_{\mathbb{H}^2} + \|\partial_t y - \partial_t \widetilde{y}\|_{L^2\left(\omega_{t_0, t_1}\right)}\right) \label{eqthm1}
\end{equation}
for all initial data $Y_0, \widetilde{Y}_0\in \mathcal{I}$ and potentials $(p,q), (\widetilde{p}, \widetilde{q}) \in \mathcal{P}$.
\end{theorem}
Consequently, we obtain the following logarithmic stability result for initial temperatures belonging to $\mathcal{I}$.
\begin{proposition}\label{pthm1}
Consider $Y$ and $\widetilde{Y}$ the corresponding solutions of $\left(\mathcal{E}_{p,q,Y_0}\right)$ and $\left(\mathcal{E}_{\widetilde{p},\widetilde{q},\widetilde{Y}_0}\right)$ respectively. There exist two positive constants $C =C\left(\Omega, \omega,t_0, t_1, r,R\right)$ and $C_1 =C_1\left(\Omega, \omega,t_0, t_1, r,R\right)$ such that, for all initial data $Y_0, \widetilde{Y}_0 \in \mathcal{I}$ and potentials $(p,q), (\widetilde{p}, \widetilde{q}) \in \mathcal{P}$,
\begin{equation}
\|Y_0 - \widetilde{Y}_0\|_{\mathbb{L}^{2}} \leq \frac{-C}{\log\left( C_1\left(\|Y(\theta, \cdot)-\widetilde{Y}(\theta, \cdot)\|_{\mathbb{H}^2} + \|\partial_t y - \partial_t \widetilde{y}\|_{L^2\left(\omega_{t_0, t_1}\right)}\right)\right)}
\end{equation}
for $\|Y(\theta, \cdot)-\widetilde{Y}(\theta, \cdot)\|_{\mathbb{H}^2} + \|\partial_t y - \partial_t \widetilde{y}\|_{L^2\left(\omega_{t_0, t_1}\right)}$ sufficiently small. 
\end{proposition}

The idea of using Carleman estimates within the framework of inverse problems was first introduced by Bukhgeim and Klibanov \cite{BK'81} in 1981. They proved uniqueness and H\"older stability results using a local Carleman estimate (estimate fulfilled by compactly supported functions). In 1996, Puel and Yamamoto modified this method using a global Carleman estimate for the wave equation in \cite{PY'96}, which led to improving the Bukhgeim-Klibanov technique by obtaining a Lipschitz stability result for the source term. It is in 1998 that the method was extended to the parabolic case by Imanuvilov and Yamamoto \cite{IY'98} for the associated inverse source problem, using Carleman estimates developed by Firsikov and Imanuvilov in the context of null controllability of classical parabolic equations \cite{FI'96}.

Most of the inverse problems encountered in the literature impose classical homogeneous boundary conditions (Dirichlet, Neumann or mixed boundary conditions), see, for instance, \cite{BFM'16, IY'98, PY'96}. Recently, the authors have studied a linear inverse source problem for a parabolic system with dynamic boundary conditions \cite{ACMO'19} and obtained a Lipschitz stability estimate for the forcing terms in the interior and on the boundary of the domain. The problem therein can be viewed as a linearized form of our inverse coefficient problem around a given solution. However, it only implies a local stability estimate where the constant $C$ in \eqref{eqthm1} will depend on initial data.

Regarding the inverse coefficient problems for parabolic equations, Yamamoto and Zou in 2001 have adapted the underlying method to prove a Lipschitz stability result for the radiative potential in the classical heat equation \cite{YZ'01}. Later on, the results were established for periodic potentials by Choi \cite{Ch'03}. In such problems, it is very common to consider smooth coefficients due to the use of maximum principles to prove positivity and boundedness results of the solution. Here we use a semigroup approach which allows us to weaken the regularity of coefficients in our inverse potentials problem.

The standard ingredient to infer logarithmic stability for initial temperature is the logarithmic convexity method. In \cite{Ch'03, YZ'01}, the authors used this method from \cite{Pa'75} for the operator $\Delta -p$ with appropriate domains. The latter result of logarithmic convexity is known for self-adjoint operators. Using an extension of the logarithmic convexity method from \cite{Is'17}, similar results were obtained in \cite{CGR'06} for a $2\times 2$ reaction-diffusion system governed by a bounded perturbation of a self-adjoint operator. The aforementioned method requires more regularity on coefficients as well as initial conditions. In our case, the operator is not self-adjoint and contains non-smooth drifts. We also deal with less regular initial data. Thus, we shall use a different and rather general logarithmic convexity result from \cite{Mi'75} (see Appendix \ref{ap1}).
 
The rest of this paper is organized as follows. In Section \ref{sec2}, we highlight the well-posedness and the positivity of the solution. Section \ref{sec3} is devoted to the proof of Lipschitz stability for the above coefficient inverse problem using Carleman estimates. Then, the logarithmic stability result for initial data is derived. Finally, we give an overview on the logarithmic convexity method in Appendix.

\section{General Framework}\label{sec2}
\subsection{Notations}
We denote the Lebesgue measure on $\Omega$ by $\d x$ and the surface measure on $\Gamma$ by $\d S$. For $t_0, t_1 \in \mathbb{R}, t_0 <t_1$, set
$$\Omega_{t_0,t_1}:=(t_0,t_1)\times \Omega, \quad \Gamma_{t_0,t_1}:=(t_0,t_1)\times \Gamma, \quad\omega_{t_0,t_1}:=(t_0,t_1)\times \omega.$$
Let us introduce the real product spaces defined by
$$\mathbb{L}^2:=L^2(\Omega, \d x)\times L^2(\Gamma, \d S)$$
and
$$\mathbb{L}^\infty=L^\infty(\Omega)\times L^\infty(\Gamma), \qquad \mathbb{L}^\infty_T=L^\infty(\Omega_T)\times L^\infty(\Gamma_T).$$
Recall that $\mathbb{L}^2$ is a real Hilbert space with the corresponding scalar product given by $\langle (y,y_\Gamma),(z,z_\Gamma)\rangle_{\mathbb{L}^2} =\langle y,z\rangle_{L^2(\Omega)} +\langle y_\Gamma,z_\Gamma\rangle_{L^2(\Gamma)}$. Moreover, $\mathbb{L}^2$ is a Hilbert lattice, and its positive cone is the product of the positive cones of $L^2(\Omega)$ and $L^2(\Gamma)$. For $\mathfrak{u}:=(u,u_\Gamma)\in \mathbb{L}^2 $, we denote by $\mathfrak{u}^+:=(u^+, u_\Gamma^+)$ and $\mathfrak{u}^- :=(u^-, u_\Gamma^-)$, where $v^+ :=\max\{v,0\}$ and $v^- :=-\min\{v,0\}$ either in $\Omega$ or on $\Gamma$. For the regularity of the solution, we introduce the following spaces
$$\mathbb{H}^k:=\{(y,y_\Gamma)\in H^k(\Omega)\times H^k(\Gamma)\colon y_{|\Gamma} =y_\Gamma \} \text{ for } k=1,2,$$
and 
$$\mathbb{E}_1(t_0,t_1):=H^1\left(t_0,t_1 ;\mathbb{L}^2\right) \cap L^2\left(t_0,t_1 ;\mathbb{H}^2\right),$$
$$\mathbb{E}_2(t_0,t_1):=H^1\left(t_0,t_1;\mathbb{H}^2\right) \cap H^2\left(t_0,t_1;\mathbb{L}^2\right).$$
In particular,
$\mathbb{E}_1 := \mathbb{E}_1(0,T)$ and $\mathbb{E}_2:= \mathbb{E}_2(0,T)$.

\subsection{Well-posedness and positivity of the solution}
The system $\left(\mathcal{E}_{p,q,Y_0}\right)$ can be rewritten as the following abstract Cauchy problem
\begin{numcases}{\text{(ACP)}\label{acp}}
\hspace{-0.1cm}\strut \partial_t Y(t)=\mathcal{A} Y(t), \quad 0<t<T, \nonumber\\
\hspace{-0.1cm} Y(0)=Y_0, \nonumber
\end{numcases}
where $Y:=(y,y_{\Gamma})$, $Y_0=(y_0, y_{0,\Gamma})$ and the linear operator $\mathcal{A} \colon D(\mathcal{A}) \subset \mathbb{L}^2 \longrightarrow \mathbb{L}^2$ is given by $D(\mathcal{A})=\mathbb{H}^2$ and $\mathcal{A}=\mathcal{A}_0+P$, where 
\begin{align}
\begin{aligned}
\mathcal{A}_0 &=\begin{pmatrix} \dv(A\nabla \cdot)+B\cdot \nabla & 0\\ -\partial_\nu^A & \dv_\Gamma(D\nabla_\Gamma \cdot) + \langle b, \nabla_\Gamma \cdot\rangle_{\Gamma} \end{pmatrix}, &&\quad D(\mathcal{A}_0)=\mathbb{H}^2,\\
P &=\begin{pmatrix} p & 0\\ 0 & q \end{pmatrix}, &&\quad D(P)=\mathbb{L}^2. \label{op1}
\end{aligned}
\end{align}
Since $P$ is a bounded operator on $\mathbb{L}^2$, the operator $\mathcal{A}$ is a bounded perturbation of $\mathcal{A}_0$. Then it suffices to study the generation results for $\mathcal{A}_0$ and recover those for $\mathcal{A}$ by perturbation arguments. Following \cite{ACMO'19}, we introduce the densely defined bilinear form given by
\begin{align*}
\mathfrak{a}_0[(y,y_\Gamma),(z,z_\Gamma)]&=\bigintsss_\Omega \left[A(x)\nabla y \cdot \nabla z -  (B(x) \cdot\nabla y)z \right] \d x + \int_\Gamma \left[\langle D(x) \nabla_\Gamma y_\Gamma, \nabla_\Gamma z_\Gamma \rangle_\Gamma - \langle b(x), \nabla_\Gamma y_{\Gamma} \rangle_{\Gamma} z_\Gamma \right] \d S,
\end{align*}
with form domain $D(\mathfrak{a}_0)=\mathbb{H}^1$ on the Hilbert space $\mathbb{L}^2$. We associate with the form $\mathfrak{a}_0$ an operator $\widetilde{\mathcal{A}}_0$ given by
\begin{align}
D(\widetilde{\mathcal{A}}_0)&:=\{(y,y_\Gamma) \in \mathbb{H}^1, \text{ there exists } (w,w_\Gamma) \in \mathbb{L}^2 \text{ such that } \nonumber\\
&\quad \mathfrak{a}_0[(y,y_\Gamma),(z,z_\Gamma)]=\langle (w,w_\Gamma), (z,z_\Gamma)\rangle_{\mathbb{L}^2} \text{ for all } (z,z_\Gamma) \in \mathbb{H}^1\}, \label{fo1}\\
\widetilde{\mathcal{A}}_0(y,y_\Gamma)&:=-(w,w_\Gamma) \qquad \text{ for all }(y,y_\Gamma) \in D(\widetilde{\mathcal{A}}_0). \label{fo2}
\end{align}
It follows from \cite[Theorem 1.52]{Ou'04} that the operator $\widetilde{\mathcal{A}}_0$ generates an analytic $C_0$-semigroup on $\mathbb{L}^2$. We recall the following results from \cite{ACMO'19,KMMR'19}, where we proved that $\mathcal{A}_0=\widetilde{\mathcal{A}}_0$ in \cite[Proposition 1]{ACMO'19}.
\begin{proposition}
\begin{enumerate}[label=(\alph*),leftmargin=*]
\item The operator $\mathcal{A}_0$ generates an analytic semigroup $(\mathcal{T}_0(t))_{t\geq 0}$ on $\mathbb{L}^2$.
\item The operator $\mathcal{A}$ generates an analytic semigroup $(\mathcal{T}(t))_{t\geq 0}$ with angle $\dfrac{\pi}{2}$ on $\mathbb{L}^2$.
\end{enumerate}\label{prop21}
\end{proposition}

\begin{proof}
The first assertion was proved in \cite[Proposition 1]{ACMO'19}, while the second claim can be obtained by mimicking the proof of \cite[Proposition 2.2]{KMMR'19} (see also \cite{Mo'10, Mo'06}).
\end{proof}
Consequently, we have the following regularity result.
\begin{proposition}\label{prop2}
For all $Y_0:=(y_0,y_{0,\Gamma})\in \mathbb{L}^2$, the unique mild solution of $\left(\mathcal{E}_{p,q,Y_0}\right)$ $Y:=(y,y_{\Gamma})\in C\left([0,T]; \mathbb{L}^2\right)$ is such that $Y\in \mathbb{E}_2(\tau, T)$ for any $\tau \in (0,T)$.
\end{proposition}
\begin{proof}
Let $\tau \in (0,T)$. Since $(\mathcal{T}(t))_{t\ge 0}$ is an analytic semigroup, the solution $Y(t)$ is twice differentiable and solves (ACP) on $(\tau, T)$ with initial data $Y(\tau)$. The analyticity of $(\mathcal{T}(t))_{t\ge 0}$ implies that $Y(\tau)\in \mathbb{H}^2$. Then \cite[Proposition 3.3, Part II, Chap. 1]{Ben'07} yields that $Y\in \mathbb{E}_1(\tau, T)$. Similarly, the function $Z(t)=\partial_t Y(t)$ solves (ACP) on $(\tau, T)$ with initial data $Z(\tau)\in \mathbb{H}^2$. The same previous argument implies that $Z\in \mathbb{E}_1(\tau, T)$, i.e., $Y \in \mathbb{E}_2(\tau, T)$.
\end{proof}

Next, we establish some positivity results of the associated semigroups for future use.
\begin{proposition}\label{proppos1}
The $C_0$-semigroups $(\mathcal{T}_0(t))_{t\geq 0}$ and $(\mathcal{T}(t))_{t\geq 0}$ generated respectively by $\mathcal{A}_0$ and $\mathcal{A}$ are positive. Moreover, the $C_0$-semigroup $(\mathcal{T}_0(t))_{t\geq 0}$ is Markovian, i.e., it is positive and $\mathcal{T}_0(t)(1,1)=(1,1)$ for all $t\geq 0$.
\end{proposition}

\begin{proof}
Since the form $\mathfrak{a}_0$ is real, by \cite[Theorem 2.6]{Ou'04} the positivity of the associated $C_0$-semigroup is equivalent to the following assertion
$$\text{For all } \mathfrak{u}\in \mathbb{H}^1, \text{ we have }\mathfrak{u}^+ \in \mathbb{H}^1 \text{ and } \mathfrak{a}_0(\mathfrak{u}^+, \mathfrak{u}^-)\leq 0.$$
Let $\mathfrak{u}=(y,y_\Gamma) \in \mathbb{H}^1$, then $\mathfrak{u}^+=(y^+,y_\Gamma^+)=(y^+,(y^+)_{|\Gamma}) \in \mathbb{H}^1$. We further have $\nabla y^+=\mathds{1}_{\{y>0\}}\nabla y$ and $\nabla y^-=-\mathds{1}_{\{y<0\}}\nabla y$ (see \cite[Lemma 7.6]{GT'77}). Hence, $\mathfrak{a}_0(\mathfrak{u}^+, \mathfrak{u}^-)=0$ and the $C_0$-semigroup $(\mathcal{T}_0(t))_{t\geq 0}$ is positive. The $C_0$-semigroup generated by the bounded diagonal operator $P=\begin{pmatrix}p  & 0\\ 0 & q\end{pmatrix}$ is given by $\mathrm{e}^{tP}=\begin{pmatrix}\mathrm{e}^{tp}  & 0\\ 0 & \mathrm{e}^{tq}\end{pmatrix}$, hence it is positive. Then, the $C_0$-semigroup $(\mathcal{T}(t))_{t\geq 0}$ generated by $\mathcal{A}$ is 
given by the Trotter product formula
$$\mathcal{T}(t)(y, y_\Gamma)=\lim_{n \to \infty} \left[\mathcal{T}_0\left(\frac{t}{n}\right) \mathrm{e}^{\frac{t}{n} P} \right]^n (y, y_\Gamma)$$
for all $t\geq 0$ and $(y, y_\Gamma) \in \mathbb{L}^2$. It follows that $(\mathcal{T}(t))_{t\geq 0}$ is also positive. To prove that $\mathcal{T}_0(t)(1,1)=(1,1)$, for all $t\geq 0$, it suffices to see that $\mathcal{A}_0 (1,1)=(0,0)$ which holds by definition of $\mathcal{A}_0$.
\end{proof}

As a direct corollary of Proposition \ref{proppos1}, we have the following result.
\begin{corollary}\label{corpos1}
Consider an initial data $(y_0, y_{0,\Gamma}) \in \mathbb{L}^2$ such that $y_0\geq 0$ on $\Omega$ and $y_{0,\Gamma} \geq 0$ on $\Gamma$. Then, the system $\left(\mathcal{E}_{p,q,Y_0}\right)$ admits a unique mild solution $Y=(y, y_\Gamma) \in C\left([0,T]; \mathbb{L}^2\right)$ such that $y \geq 0$ on $\Omega_T$ and $y_\Gamma \geq 0$ on $\Gamma_T$. 
\end{corollary}

Henceforth, all the inequalities between measurable functions are often understood in the ``almost everywhere$"$ sense. For fixed positive constants $r$ and $R$, we recall the set of admissible initial data
\begin{align}
\mathcal{I}:=\{Y_0=(y_0, y_{0,\Gamma}) \in \mathbb{W}^{2,\infty} \colon y_0, y_{0,\Gamma}\geq r, \|Y_0\|_{2,\infty} \leq R\} \label{init1}\
\end{align}
and the set of admissible potentials
\begin{align}
\mathcal{P}:=\{(p,q) \in \mathbb{L}^\infty \colon \|p\|_\infty ,\|q\|_\infty \leq R\}.
\end{align}

In order to control some terms in the proof of the stability estimate, we shall use the following positivity and boundedness results.
\begin{lemma}\label{lempos1}
Let $Y_0=(y_0, y_{0,\Gamma}) \in \mathcal{I}$ and $(p,q) \in \mathcal{P}$. Then, the following assertions hold
\begin{enumerate}[label=(\roman*),leftmargin=*]
\item The solution $Y=(y,y_\Gamma)$ of $\left(\mathcal{E}_{p,q,Y_0}\right)$ satisfies
$$y(t,x) \geq \mathrm{e}^{-Rt} r >0 \text{ on } \Omega_T \quad\text{ and } \quad y_\Gamma(t,x) \geq \mathrm{e}^{-Rt} r >0 \text{ on } \Gamma_T .$$
\item There exists a positive constant $C=C\left(T, R\right)$ so that the solution of $\left(\mathcal{E}_{p,q,Y_0}\right)$ satisfies $\partial_t Y\in \mathbb{L}^{\infty}_T$ and
\begin{equation}
\|\partial_t Y\|_{\mathbb{L}^\infty_T} \leq C. \label{bd1}
\end{equation}
\end{enumerate}
\end{lemma}

\begin{proof}
$(i)$ It is clear that $Z:=(z,z_\Gamma)=\mathrm{e}^{Rt} Y=\left(\mathrm{e}^{Rt}y,\mathrm{e}^{Rt} y_\Gamma\right)$ is the solution of
\begin{empheq}[left = \empheqlbrace]{alignat=2}
\begin{aligned}
&\partial_t z =\dv(A(x) \nabla z) + B(x)\cdot \nabla z +(R+p(x))z, &\text{ in } \Omega_T ,\\
&\partial_t z_{\Gamma}=\dv_{\Gamma} (D(x)\nabla_\Gamma z_{\Gamma}) -\partial_{\nu}^A z + \langle b(x), \nabla_\Gamma z_{\Gamma} \rangle_{\Gamma} +(R+q(x)) z_{\Gamma}, &\text{on } \Gamma_T,\\
&z_{\Gamma}(t,x)= z_{|\Gamma}(t,x), &\text{ on } \Gamma_T,\\
&(z,z_{\Gamma})\rvert_{t=0}=(y_0, y_{0,\Gamma}), &\Omega\times\Gamma.
\end{aligned}
\end{empheq}
By Corollary \ref{corpos1}, $Z \geq 0$. Since $R+p \geq 0$ on $\Omega$, $R+q \geq 0$ on $\Gamma$ and the $C_0$-semigroup $(\mathcal{T}_0(t))_{t\geq 0}$ is Markovian, by Proposition \ref{proppos1}, it follows that
\begin{align*}
Z(t) &=\mathcal{T}_0(t)Y_0 +\int_0^t \mathcal{T}_0(t-\tau)[(R+p)z(\tau), (R+q)z_\Gamma(\tau)] \,\d \tau \\
& \geq \mathcal{T}_0(t)Y_0 \\
& \geq \mathcal{T}_0(t)(r, r) = r \mathcal{T}_0(t)(1,1)=(r, r).
\end{align*}
Hence, $Y(t)=\mathrm{e}^{-Rt} Z(t) \geq \mathrm{e}^{-Rt}(r, r)$.

$(ii)$ It suffices to prove that, for any $Y_0 \in \mathbb{L}^{\infty}$, we have $Y\in \mathbb{L}^{\infty}_T$ and there exists a positive constant $C=C\left(T, R\right)$ such that 
\begin{equation}
\|Y\|_{\mathbb{L}^\infty_T} \leq C \|Y_0\|_{\mathbb{L}^{\infty}}. \label{bd2}
\end{equation}

Since $p-R \leq 0$ on $\Omega$ and $q-R \leq 0$ on $\Gamma$, for all $(u,u_\Gamma) \in \mathbb{L}^2$ such that $u \ge 0$ on $\Omega$ and $u_\Gamma \ge 0$ on $\Gamma$, we have
\begin{align*}
\mathrm{e}^{\frac{t}{n} \left(P-R I_{\mathbb{L}^2}\right)} (u,u_\Gamma) & = \begin{pmatrix}\mathrm{e}^{\frac{t}{n}(p-R)}  & 0\\ 0 & \mathrm{e}^{\frac{t}{n}(q-R)}\end{pmatrix} (u,u_\Gamma) \\
& = \left(\mathrm{e}^{\frac{t}{n}(p-R)} u, \mathrm{e}^{\frac{t}{n}(q-R)}u_{\Gamma}\right) \\
& \le (u,u_\Gamma)
\end{align*}
for all $n\in \mathbb{N}$ and all $t\in [0,T]$. Then the positivity of $(\mathcal{T}_0(t))_{t\geq 0}$ implies that $$\mathcal{T}_{0}\left(\frac{t}{n}\right) \mathrm{e}^{\frac{t}{n} \left(P-R I_{\mathbb{L}^2}\right)} (y_0, y_{0,\Gamma}) \le \mathcal{T}_{0}\left(\frac{t}{n}\right) (y_0, y_{0,\Gamma}).$$
Iterating the process, we obtain that $$\left[\mathcal{T}_{0}\left(\frac{t}{n}\right) \mathrm{e}^{\frac{t}{n}\left(P-R I_{\mathbb{L}^2}\right)}\right]^{n}\left(y_{0}, y_{0, \Gamma}\right) \leq \left[\mathcal{T}_{0}\left(\frac{t}{n}\right)\right]^{n}\left(y_{0}, y_{0, \Gamma}\right)=\mathcal{T}_0(t)(y_0, y_{0,\Gamma})$$
for all $n\in \mathbb{N}$ and all $t\in [0,T]$.
By the Trotter product formula, for all $t\in [0,T]$, we deduce
\begin{align*}
\mathrm{e}^{-Rt} Y(t)&=\lim_{n \to \infty} \left[\mathcal{T}_0\left(\frac{t}{n}\right) \mathrm{e}^{\frac{t}{n} \left(P-R I_{\mathbb{L}^2}\right)} \right]^n (y_0, y_{0,\Gamma})\\
& \leq \mathcal{T}_0(t)(y_0, y_{0,\Gamma})\\
& \leq \mathcal{T}_0(t)(\|Y_0\|_{\mathbb{L}^{\infty}}, \|Y_0\|_{\mathbb{L}^{\infty}})\\
& = \|Y_0\|_{\mathbb{L}^{\infty}} \mathcal{T}_0(t)(1,1)=\left(\|Y_0\|_{\mathbb{L}^{\infty}}, \|Y_0\|_{\mathbb{L}^{\infty}}\right).
\end{align*}
Finally, $Y(t)\leq \mathrm{e}^{RT} \left(\|Y_0\|_{\mathbb{L}^{\infty}}, \|Y_0\|_{\mathbb{L}^{\infty}}\right)$ for all $t\in [0,T]$.
\end{proof}

\section{Stability estimates for radiative potentials and initial data}\label{sec3}
\subsection{Carleman estimate}
We start by recalling the following lemma from \cite{FI'96}, which is a key tool to construct the weight functions needed in the Carleman estimate.
\begin{lemma}\label{lm1}
Let $\omega \Subset \Omega$ be a nonempty open subset. Then, there exists a function $\eta^0\in C^2(\overline{\Omega})$ such that       
\begin{align*}
\eta^0 &> 0  \;\; \text{ in } \Omega \quad\text{ and }\quad|\nabla\eta^0|  > 0 \;\;\, \qquad\text{ in } \overline{\Omega\backslash\omega},\\
\eta^0 &=0 \;\;\text{ on } \Gamma \quad\text{ and } \quad\partial_\nu \eta^0 \leq -c <0 \;\; \text{ on } \Gamma
\end{align*}
for some constant $c>0$.
\end{lemma}
Let $t_0, t_1 \in \mathbb{R}$, $0<t_0<t_1 \leq T$. Consider the following weight functions
\begin{align*}
\alpha(t,x)&=\frac{\mathrm{e}^{2\lambda \|\eta^0\|_\infty}- \mathrm{e}^{\lambda \eta^0(x)}}{\gamma(t)},\\
\xi(t,x)&=\frac{\mathrm{e}^{\lambda \eta^0(x)}}{\gamma(t)}, \qquad \gamma(t)=(t-t_0)(t_1-t)
\end{align*}
for all $(t,x)\in \overline{\Omega}_{t_0, t_1}$, where $\lambda \geq 1$ is a large parameter (to fix later) which depends on $\Omega$ and $\omega$. In the following lemma, we collect some properties of $\alpha$ and $\xi$ that will be useful in the sequel (see \cite{ACMO'19}).
\begin{lemma}
\begin{enumerate}[label=(\alph*),leftmargin=*] The functions $\alpha$ and $\xi$ satisfy the following properties.
\item $\alpha$ and $\xi$ are positive on $\overline{\Omega}_{t_0, t_1}:=(t_0,t_1)\times \overline{\Omega}$.
\item $|\partial_t \alpha| \leq CT\xi^2$ and $\nabla_\Gamma \alpha =\nabla_\Gamma \xi =0 \; \text{ on }\Gamma$.
\item $\xi \geq \dfrac{4}{(t_1 -t_0)^2}$ and $\xi \leq \dfrac{(t_1 -t_0)^4}{16} \xi^3$.
\item $t\mapsto\alpha(t, \cdot)$ attains its minimum at $\theta=\dfrac{t_0 +t_1}{2}$.
\item For $s>0$, $\inf\limits_{x\in \Omega} \mathrm{e}^{-2s\alpha(\theta, x)} >0$ and $\sup\limits_{(t,x)\in \overline{\Omega}_{t_0, t_1}} \mathrm{e}^{-2s\alpha(t,x)} \xi^3(t,x) < +\infty$.
\end{enumerate}
\end{lemma}

For $(z,z_\Gamma) \in \mathbb{E}_1$, we consider the following differential operators
\begin{align*}
L z &:=\partial_t z -\dv (A(x) \nabla z) - B(x)\cdot \nabla z &&\quad\text{ in }\Omega_T,\\
L_\Gamma (z_\Gamma, z) &:=\partial_t z_\Gamma -\dv_\Gamma (D(x)\nabla_\Gamma z_\Gamma) + \partial_\nu^A z - \langle b(x), \nabla_\Gamma z_{\Gamma} \rangle_{\Gamma}  &&\quad\text{ on }\Gamma_T.
\end{align*}
Following the decomposition in \cite{ACMO'19}, we set $\sigma:=A(\cdot)\nabla \eta^0 \cdot \nabla\eta^0$ and
\begin{align}
M_1 z &:=\partial_t z + 2s\lambda \xi A(x)\nabla \eta^0\cdot \nabla z + 2s\lambda^2 \xi \sigma z, \label{dec1}\\
N_1 z_\Gamma &:= \partial_t z_\Gamma - s\lambda \xi z_\Gamma \partial_\nu^A \eta^0. \label{dec2}
\end{align}
We will use the following lemma (see \cite[Lemma 2.4]{ACMO'19}), which is a key tool to prove the Lipschitz stability for our inverse problem.
\begin{lemma}[Carleman estimate]\label{lm2}
Let $T>0$, $\omega\Subset \Omega$ be nonempty and open subset. Then, there are three positive constants $\lambda_1,s_1 \geq 1$ and $C>0$ such that, for any $\lambda\geq \lambda_1$ and $s\geq s_1$, the following inequality holds
\begin{align}
& \bigintssss_{\Omega_{t_0, t_1}} \left(\frac{1}{s\xi} \left(|\partial_t z|^2 + |\dv(A\nabla z)|^2 \right)+ s\lambda^2 \xi |\nabla z|^2 + s^3\lambda^4\xi^3 |z|^2 \right)\mathrm{e}^{-2s\alpha} \,\d x\,\d t \quad  \nonumber\\
& + \int_{\Gamma_{t_0, t_1}} \left(\frac{1}{s\xi} (|\partial_t z_\Gamma|^2 +|\dv(D\nabla_\Gamma z_\Gamma)|^2)+s\lambda \xi |\nabla_\Gamma z_\Gamma|^2 + s^3\lambda^3\xi^3 |z_\Gamma|^2 + s\lambda \xi |\partial_\nu^A z|^2\right)\mathrm{e}^{-2s\alpha} \,\d S\,\d t \nonumber\\
& \;\;\leq  C \left(s^3\lambda^4\int_{\omega_{t_0, t_1}} \mathrm{e}^{-2s\alpha} \xi^3 |z|^2 \,\d x\,\d t + \int_{\Omega_{t_0, t_1}}\mathrm{e}^{-2s\alpha}|Lz|^2 \,\d x\,\d t + \int_{\Gamma_{t_0, t_1}}\mathrm{e}^{-2s\alpha}|L_\Gamma (z_\Gamma, z)|^2 \,\d S\,\d t\right) \label{car1}
\end{align}
for all $(z,z_\Gamma)\in \mathbb{E}_1$ and, for $(y,y_\Gamma):=\mathrm{e}^{-s\alpha} (z,z_\Gamma)$, we also have
\begin{align}
& \| M_1 y\|^2_{L^2\left(\Omega_{t_0, t_1}\right)} +\| N_1 y_\Gamma\|^2_{L^2\left(\Gamma_{t_0, t_1}\right)} +\, s^3 \lambda^4 \int_{\Omega_{t_0, t_1}} \mathrm{e}^{-2s\alpha} \xi^3 |z|^2 \,\d x\,\d t + s^3 \lambda^3 \int_{\Gamma_{t_0, t_1}} \mathrm{e}^{-2s\alpha} \xi^3 |z_\Gamma|^2 \,\d S\,\d t \nonumber\\
&\leq C \left(s^3 \lambda^4 \int_{\omega_{t_0, t_1}} \mathrm{e}^{-2s\alpha} \xi^3 |z|^2 \,\d x\,\d t + \int_{\Omega_{t_0, t_1}}\mathrm{e}^{-2s\alpha}|L z|^2 \,\d x\,\d t + \int_{\Gamma_{t_0, t_1}}\mathrm{e}^{-2s\alpha}|L_\Gamma (z_\Gamma, z)|^2 \,\d S\,\d t \right). \label{el5s1}
\end{align}
\end{lemma}

\subsection{Proof of the stability estimate for potentials \label{subsec3.2}}
Now, we are ready to prove the main result on Lipschitz stability for our inverse potentials problem stated in Theorem \ref{thm1}, by using some ideas from \cite{CGR'06, YZ'01} in a modified form.

Let $Z=Y-\widetilde{Y}$, $a=p-\widetilde{p}$ and $\ell=q-\widetilde{q}$. Then,
\begin{empheq}[left = \empheqlbrace]{alignat=2}
\begin{aligned}
&\partial_t z =\dv(A(x) \nabla z) + B(x)\cdot \nabla z +p(x)z+a(x)\widetilde{y}, &\text{ in } \Omega_T ,\\
&\partial_t z_{\Gamma}=\dv_{\Gamma} (D(x)\nabla_\Gamma z_{\Gamma}) -\partial_{\nu}^A z + \langle b(x), \nabla_\Gamma z_{\Gamma} \rangle_{\Gamma}+q(x) z_{\Gamma}+\ell(x)\widetilde{y}_{\Gamma}, &\text{on } \Gamma_T,\\
&z_{\Gamma}(t,x)= z_{|\Gamma}(t,x), &\text{ on } \Gamma_T , \nonumber
\end{aligned}
\end{empheq}
and $V=\partial_t Z$ is the solution of
\begin{empheq}[left = \empheqlbrace]{alignat=2}
\begin{aligned}
&\partial_t v =\dv(A(x) \nabla v) + B(x)\cdot \nabla v +p(x) v+a(x)\partial_t\widetilde{y}, &\text{ in } \Omega_T ,\\
&\partial_t v_{\Gamma}=\dv_{\Gamma} (D(x)\nabla_\Gamma v_{\Gamma}) -\partial_{\nu}^A v + \langle b(x), \nabla_\Gamma v_{\Gamma} \rangle_{\Gamma}+q(x) z_{\Gamma}+\ell(x)\partial_t\widetilde{y}_{\Gamma}, &\text{on } \Gamma_T,\\
&v_{\Gamma}(t,x)= v_{|\Gamma}(t,x), &\text{ on } \Gamma_T . \label{el0s0}
\end{aligned}
\end{empheq}

We set
$$U=\mathrm{e}^{-s\alpha} V, \quad I:=\int_{\Omega_{t_0,\theta}} M_1 u\, u \,\d x\,\d t, \quad J:=\int_{\Gamma_{t_0,\theta}} N_1 u_\Gamma\, u_\Gamma \,\d S\,\d t,$$
where $M_1$ and $N_1$ are defined by \eqref{dec1} and \eqref{dec2} respectively. Since $U\in \mathbb{E}_1(t_0, T)$, $(M_1 u, N_1 u_\Gamma)\in L^2\left(t_0, T; \mathbb{L}^2\right)$, and then the integrals $I$ and $J$ are well defined.

For the sake of simplicity, we will first prove some lemmas that are needed in the proof of Theorem \ref{thm1}. Henceforth, $C=C\left(\Omega, \omega,t_0, t_1, r, R\right)$ denotes a positive constant which may vary from line to line.

\begin{lemma}\label{lems1}
There exist $\lambda_1, s_1 \geq 1$ and a positive constant $C$ so that for all $\lambda \geq \lambda_1$ and $s\geq s_1$, if $Y$ and $\widetilde{Y}$ are the solutions of $\left(\mathcal{E}_{p,q,Y_0}\right)$ and $\left(\mathcal{E}_{\widetilde{p},\widetilde{q},\widetilde{Y}_0}\right)$ respectively, we have
\begin{align}
|I|+|J| &\leq C \left[s^{3/2} \lambda ^{5/2} \int_{\omega_{t_0, t_1}} \mathrm{e}^{-2s\alpha} \xi^3 v^2 \,\d x\,\d t +s^{-3/2} \lambda^{-3/2} \int_{\Omega_{t_0, t_1}} \mathrm{e}^{-2s\alpha} a^2 (\partial_t \widetilde{y})^2 \,\d x\,\d t \right. \nonumber\\
& \quad\qquad \left. +\, s^{-3/2} \lambda^{-3/2} \int_{\Gamma_{t_0, t_1}} \mathrm{e}^{-2s\alpha} \ell^2 (\partial_t \widetilde{y}_\Gamma)^2 \,\d S\,\d t \right]. \label{eqlems1}
\end{align}
\end{lemma}

\begin{proof}
Using Cauchy-Schwarz and Young inequalities, we have
\begin{align}
|I| &\leq s^{-3/2} \lambda^{-2}\left(\int_{\Omega_{t_0, \theta}} (M_1 u)^2 \,\d x\,\d t \right )^{1/2} \left( s^{3} \lambda ^{4} \int_{\Omega_{t_0,\theta}} \mathrm{e}^{-2s\alpha} v^2 \,\d x\,\d t \right)^{1/2} \nonumber\\
&\leq \frac{1}{2}s^{-3/2} \lambda^{-3/2} \left(\| M_1 u\|^2_{L^2\left(\Omega_{t_0, t_1}\right)}+ C_1 s^3 \lambda^4 \int_{\Omega_{t_0, t_1}} \mathrm{e}^{-2s\alpha} \xi^3 v^2 \,\d x\,\d t \right)\label{el2s1}
\end{align}
for $\lambda \ge 1$, using $\xi^3 \geq \dfrac{1}{C_1} >0$ where $C_1 =\frac{(t_1-t_0)^6}{64}$. Similarly,
\begin{align}
|J| &\leq s^{-3/2} \lambda^{-3/2}\left(\int_{\Gamma_{t_0, \theta}} (N_1 u_\Gamma)^2 \,\d S\,\d t \right )^{1/2} \left( s^{3} \lambda^{3} \int_{\Gamma_{t_0,\theta}} \mathrm{e}^{-2s\alpha} v_\Gamma^2 \,\d S\,\d t \right)^{1/2} \nonumber\\
&\leq \frac{1}{2}s^{-3/2} \lambda^{-3/2} \left(\| N_1 u_\Gamma\|^2_{L^2\left(\Gamma_{t_0, t_1}\right)}+ C_1 s^3 \lambda^3 \int_{\Gamma_{t_0, t_1}} \mathrm{e}^{-2s\alpha} \xi^3 v_\Gamma^2 \,\d S\,\d t \right).\label{el4s1}
\end{align}
Adding up \eqref{el2s1} and \eqref{el4s1}, we obtain
\begin{align}
|I|+|J| &\leq \frac{s^{-3/2}\lambda^{-3/2}}{2} \left( \| M_1 u\|^2_{L^2\left(\Omega_{t_0, t_1}\right)} +\| N_1 u_\Gamma\|^2_{L^2\left(\Gamma_{t_0, t_1}\right)} +\, C_1 s^3 \lambda^4 \int_{\Omega_{t_0, t_1}} \mathrm{e}^{-2s\alpha} \xi^3 v^2 \,\d x\,\d t \right.\nonumber\\
& \hspace{3cm} \left. +\ C_1 s^3 \lambda^3 \int_{\Gamma_{t_0, t_1}} \mathrm{e}^{-2s\alpha} \xi^3 v_\Gamma^2 \,\d S\,\d t \right). \label{ij}
\end{align}
By applying Carleman estimate \eqref{el5s1} to \eqref{el0s0}, we have
\begin{align}
& \| M_1 u\|^2_{L^2\left(\Omega_{t_0, t_1}\right)} +\| N_1 u_\Gamma\|^2_{L^2\left(\Gamma_{t_0, t_1}\right)} +\, s^3 \lambda^4 \int_{\Omega_{t_0, t_1}} \mathrm{e}^{-2s\alpha} \xi^3 v^2 \,\d x\,\d t +\, s^3 \lambda^3 \int_{\Gamma_{t_0, t_1}} \mathrm{e}^{-2s\alpha} \xi^3 v_\Gamma^2 \,\d S\,\d t \nonumber\\
&\leq C \left[s^3 \lambda^4 \int_{\omega_{t_0, t_1}} \mathrm{e}^{-2s\alpha} \xi^3 v^2 \,\d x\,\d t + \int_{\Omega_{t_0, t_1}} \mathrm{e}^{-2s\alpha} [p^2 v^2 + a^2 (\partial_t \widetilde{y})^2] \,\d x\,\d t \right. \nonumber\\
& \qquad\qquad \left. +\, \int_{\Gamma_{t_0, t_1}} \mathrm{e}^{-2s\alpha} [q^2 v_\Gamma^2 + \ell^2 (\partial_t \widetilde{y}_\Gamma)^2] \,\d S\,\d t \right]. \label{el5s11}
\end{align}
Since $1\leq C_1\xi^3$ and the potentials $p$ and $q$ are bounded, we have
\begin{align*}
& \int_{\Omega_{t_0, t_1}}  \mathrm{e}^{-2s\alpha} p^2 v^2 \,\d x\,\d t +\int_{\Gamma_{t_0, t_1}} \mathrm{e}^{-2s\alpha} q^2 v_\Gamma^2  \,\d S\,\d t \leq C \left[ \int_{\Omega_{t_0, t_1}} \mathrm{e}^{-2s\alpha} \xi^3 v^2 \,\d x\,\d t +\int_{\Gamma_{t_0, t_1}} \mathrm{e}^{-2s\alpha} \xi^3 v_\Gamma^2  \,\d S\,\d t \right],
\end{align*}
and this later term can be absorbed on the left hand side of \eqref{el5s11} for $s$ and $\lambda$ large enough. Combining this with \eqref{ij}, we obtain \eqref{eqlems1}.
\end{proof}

\begin{lemma}\label{lems2}
There exist $\lambda_1, s_1 \geq 1$ and $C>0$ such that, for all $\lambda \geq \lambda_1$ and $s\geq s_1$, we have
\begin{align}
&\int_\Omega \mathrm{e}^{-2s\alpha(\theta, x)} a^2(x)[\widetilde{y}(\theta, x)]^2 \,\d x -C s^{-3/2} \lambda^{-3/2} \int_{\Omega_{t_0, t_1}} \mathrm{e}^{-2s\alpha} a^2(x)(\partial_t \widetilde{y})^2 \,\d x\,\d t \nonumber\\
& + \int_\Gamma \mathrm{e}^{-2s\alpha(\theta, x)} \ell^2(x)[\widetilde{y}_\Gamma(\theta, x)]^2 \,\d x -C s^{-3/2} \lambda^{-3/2} \int_{\Gamma_{t_0, t_1}} \mathrm{e}^{-2s\alpha} \ell^2(x) (\partial_t \widetilde{y}_\Gamma)^2 \,\d S\,\d t \nonumber \\
& \qquad \leq C \left[s^{3/2} \lambda^{5/2} \int_{\omega_{t_0, t_1}}\mathrm{e}^{-2s\alpha}\xi^3 v^2 \,\d x\,\d t + \|Z(\theta, \cdot)\|_{\mathbb{H}^2}^2\right] . \label{eqlems2}
\end{align}
\end{lemma}

\begin{proof}
We have $$v(\theta, x)=\partial_t z(\theta,x)=\dv(A(x) \nabla z(\theta,x))+ B(x)\cdot \nabla z(\theta,x)+p(x)z(\theta,x)+a(x)\widetilde{y}(\theta,x).$$
Since $A\in C^1(\overline{\Omega},\mathbb{R}^{N\times N})$ and $B, \; p$ are bounded, we obtain
\begin{align}
\int_\Omega \mathrm{e}^{-2s\alpha(\theta, x)} a^2 [\widetilde{y}(\theta, x)]^2 \,\d x &\leq C \int_\Omega \mathrm{e}^{-2s\alpha(\theta, x)} v^2(\theta, x) \,\d x + C \|z(\theta,\cdot)\|_{H^2(\Omega)}^2 .\label{e1s2}
\end{align}
Similarly,
\begin{equation*}
v_\Gamma(\theta, x)=\dv_{\Gamma} (D\nabla_\Gamma z_{\Gamma}(\theta,x))-\partial_{\nu}^A z(\theta,x)+ \langle b(x), \nabla_\Gamma z_{\Gamma}(\theta,x)\rangle_{\Gamma}+q(x)z_\Gamma(\theta,x)+\ell(x)\widetilde{y}_\Gamma(\theta,x).
\end{equation*}
By trace theorem (see, e.g., \cite{LM'72}), we have $\|\partial_\nu^A z(\theta,\cdot)\|_{L^2(\Gamma)}\leq C\|z(\theta, \cdot)\|_{H^2(\Omega)}$. Since $D\in C^1\left(\Gamma,\mathbb{R}^{N\times N}\right)$, the boundedness of $b$ and $q$ implies that
\begin{align}
\int_\Gamma \mathrm{e}^{-2s\alpha(\theta, x)} \ell^2 [\widetilde{y}_\Gamma(\theta, x)]^2 \,\d S &\leq C \int_\Gamma \mathrm{e}^{-2s\alpha(\theta, x)} v^2_\Gamma(\theta, x) \,\d S + C\|z_\Gamma(\theta,\cdot)\|_{H^2(\Gamma)}^2 +C \|z(\theta,\cdot)\|_{H^2(\Omega)}^2 .\label{e3s2}
\end{align}
Combining \eqref{e1s2} and \eqref{e3s2}, we obtain
\begin{align}
&\int_\Omega \mathrm{e}^{-2s\alpha(\theta, x)} a^2 [\widetilde{y}(\theta, x)]^2 \,\d x +\int_\Gamma \mathrm{e}^{-2s\alpha(\theta, x)} \ell^2 [\widetilde{y}_\Gamma(\theta, x)]^2 \,\d S \nonumber\\
& \quad \leq C \left[\int_\Omega \mathrm{e}^{-2s\alpha(\theta, x)} v^2(\theta, x) \,\d x +\int_\Gamma \mathrm{e}^{-2s\alpha(\theta, x)} v^2_\Gamma(\theta, x) \,\d S +\|Z(\theta,\cdot)\|_{\mathbb{H}^2}^2\right]. \label{e4s2}
\end{align}
Next, we estimate the terms $\displaystyle \int_\Omega \mathrm{e}^{-2s\alpha(\theta, x)} v^2(\theta, x) \,\d x$ and $\displaystyle\int_\Gamma \mathrm{e}^{-2s\alpha(\theta, x)} v^2_\Gamma(\theta, x) \,\d S $.\\
We have
$$I=\frac{1}{2} \int_{\Omega_{t_0, \theta}} \partial_t (u^2) \,\d x\,\d t +2 s\lambda \int_{\Omega_{t_0, \theta}} \xi A(x)\nabla \eta^0 \cdot \nabla u \, u \,\d x\,\d t +2s\lambda^2 \int_{\Omega_{t_0, \theta}} \xi \sigma u^2 \,\d x\,\d t.$$
Integration by parts over $\Omega$ yields
\begin{align*}
I &=\frac{1}{2} \int_{\Omega_{t_0, \theta}} \partial_t (u^2) \,\d x\,\d t - s\lambda^2 \int_{\Omega_{t_0, \theta}} \xi \sigma u^2 \,\d x\,\d t - s\lambda \int_{\Omega_{t_0, \theta}} \xi \dv(A(x)\nabla \eta^0) u^2 \,\d x\,\d t \\
&\qquad + s\lambda \int_{\Gamma_{t_0, \theta}} \xi \partial_\nu^A \eta^0 u_\Gamma^2 \,\d S\,\d t +2s\lambda^2 \int_{\Omega_{t_0, \theta}} \xi \sigma u^2 \,\d x\,\d t\\
&= \frac{1}{2} \int_{\Omega_{t_0, \theta}} \partial_t (u^2) \,\d x\,\d t + s\lambda^2 \int_{\Omega_{t_0, \theta}} \xi \sigma u^2 \,\d x\,\d t - s\lambda \int_{\Omega_{t_0, \theta}} \xi \dv(A(x)\nabla \eta^0) u^2 \,\d x\,\d t \\
&\qquad + s\lambda \int_{\Gamma_{t_0, \theta}} \xi \partial_\nu^A \eta^0 u_\Gamma^2 \,\d S\,\d t,
\end{align*}
where we employed $\nabla \xi = \lambda \xi \nabla \eta^0$. Since $u(t_0, x)=\lim\limits_{t \to t_0} \mathrm{e}^{-s\alpha(t,x)} v(t,x)=0$, we obtain
\begin{align*}
\frac{1}{2} \int_\Omega u^2(\theta,x) \,\d x &= I - s\lambda^2 \int_{\Omega_{t_0, \theta}} \xi \sigma u^2 \,\d x\,\d t + s\lambda \int_{\Omega_{t_0, \theta}} \xi \dv(A(x)\nabla \eta^0) u^2 \,\d x\,\d t - s\lambda \int_{\Gamma_{t_0, \theta}} \xi \partial_\nu^A \eta^0 u_\Gamma^2 \,\d S\,\d t.
\end{align*}
Using $u=\mathrm{e}^{-s\alpha} v$ and $\eta^0 \in C^2(\overline{\Omega})$, we derive the following inequality
\begin{align}
\int_\Omega \mathrm{e}^{-2s\alpha(\theta, x)} v^2(\theta, x) \,\d x &\leq 2 |I| + C s\lambda^2 \int_{\Omega_{t_0, \theta}} \mathrm{e}^{-2s\alpha} \xi v^2 \,\d x\,\d t +C s\lambda \int_{\Gamma_{t_0, \theta}} \mathrm{e}^{-2s\alpha} \xi v^2_\Gamma \,\d S\,\d t. \label{e2s2}
\end{align}
Analogously, we have
$$J=\frac{1}{2} \int_{\Gamma_{t_0, \theta}} \partial_t (u_\Gamma^2) \,\d S\,\d t - s\lambda \int_{\Gamma_{t_0, \theta}} \xi (\partial_\nu^A \eta^0) u_\Gamma^2 \,\d S\,\d t,$$
and $u_\Gamma(t_0, x)=0$. Then,
$$\frac{1}{2} \int_\Gamma u^2_\Gamma(\theta,x) \,\d S = J + s\lambda \int_{\Gamma_{t_0, \theta}} \xi (\partial_\nu^A \eta^0) u_\Gamma^2 \,\d S\,\d t,$$
and
\begin{align}
\int_\Gamma \mathrm{e}^{-2s\alpha(\theta, x)} v^2_\Gamma(\theta, x) \,\d S \leq 2 |J| + C s\lambda \int_{\Gamma_{t_0, \theta}} \mathrm{e}^{-2s\alpha} \xi v^2_\Gamma \,\d S\,\d t. \label{ee2s2}
\end{align}
Adding up \eqref{e2s2} and \eqref{ee2s2}, we obtain
\begin{align}
& \int_\Omega \mathrm{e}^{-2s\alpha(\theta, x)} v^2(\theta, x) \,\d x +\int_\Gamma \mathrm{e}^{-2s\alpha(\theta, x)} v^2_\Gamma(\theta, x) \,\d S \nonumber \\
& \qquad \leq 2(|I|+|J|) + C s\lambda^2 \int_{\Omega_{t_0, \theta}} \mathrm{e}^{-2s\alpha} \xi v^2 \,\d x\,\d t + C s\lambda \int_{\Gamma_{t_0, \theta}} \mathrm{e}^{-2s\alpha} \xi v^2_\Gamma \,\d S\,\d t \label{e5s2}.
\end{align}
Using Carleman estimate \eqref{car1} and $\xi \leq C_2 \xi^3$ where $C_2=\frac{(t_1-t_0)^4}{16}$, we derive
\begin{align}
& s\lambda^2 \int_{\Omega_{t_0, \theta}} \mathrm{e}^{-2s\alpha} \xi v^2 \,\d x\,\d t + s\lambda \int_{\Gamma_{t_0, \theta}} \mathrm{e}^{-2s\alpha} \xi v^2_\Gamma \,\d S\,\d t \nonumber \\
& \quad \leq C s^{-2} \lambda^{-2} \left(s^3\lambda^4 \int_{\Omega_{t_0, t_1}} \mathrm{e}^{-2s\alpha} \xi^3 v^2 \,\d x\,\d t + s^3\lambda^3 \int_{\Gamma_{t_0, t_1}} \mathrm{e}^{-2s\alpha} \xi^3  v^2_\Gamma \,\d S\,\d t \right) \nonumber \\
& \quad \leq C \left[ s\lambda^2 \int_{\omega_{t_0, t_1}} \mathrm{e}^{-2s\alpha} \xi^3 v^2 \,\d x\,\d t + s^{-2} \lambda^{-2} \int_{\Omega_{t_0, t_1}} \mathrm{e}^{-2s\alpha} [p^2 v^2 + a^2 (\partial_t \widetilde{y})^2] \,\d x\,\d t \right. \nonumber \\
& \qquad\qquad + \left.  s^{-2} \lambda^{-2} \int_{\Gamma_{t_0, t_1}} \mathrm{e}^{-2s\alpha} [q^2 v^2_\Gamma + \ell^2 (\partial_t \widetilde{y}_\Gamma)^2] \,\d S\,\d t \right]. \label{e6s2}
\end{align}
Since $\|p\|_\infty , \|q\|_\infty \leq R$ and $1 \leq C_3 \xi$ with $C_3=\frac{(t_1-t_0)^2}{4}$, we have
\begin{align*}
& s^{-2} \lambda^{-2} \int_{\Omega_{t_0, t_1}} \mathrm{e}^{-2s\alpha} p^2 v^2 \,\d x\,\d t +   s^{-2} \lambda^{-2} \int_{\Gamma_{t_0, t_1}} \mathrm{e}^{-2s\alpha} q^2 v^2_\Gamma \,\d S\,\d t \\
& \quad \leq C s^{-2} \lambda^{-2} \int_{\Omega_{t_0, t_1}} \mathrm{e}^{-2s\alpha} \xi v^2 \,\d x\,\d t + C s^{-2} \lambda^{-2} \int_{\Gamma_{t_0, t_1}} \mathrm{e}^{-2s\alpha} \xi v^2_\Gamma \,\d S\,\d t,
\end{align*}
and this later term can be absorbed by the left-hand side of \eqref{e6s2}. Thus, we arrive at
\begin{align}
& s\lambda^2 \int_{\Omega_{t_0, \theta}} \mathrm{e}^{-2s\alpha} \xi v^2 \,\d x\,\d t + s\lambda \int_{\Gamma_{t_0, \theta}} \mathrm{e}^{-2s\alpha} \xi v^2_\Gamma \,\d S\,\d t \leq C \left[ s\lambda^2 \int_{\omega_{t_0, t_1}} \mathrm{e}^{-2s\alpha} \xi^3 v^2 \,\d x\,\d t \right. \nonumber\\
& \quad \left. +\, s^{-2} \lambda^{-2} \int_{\Omega_{t_0, t_1}} \mathrm{e}^{-2s\alpha} a^2 (\partial_t \widetilde{y})^2 \,\d x\,\d t + s^{-2} \lambda^{-2} \int_{\Gamma_{t_0, t_1}} \mathrm{e}^{-2s\alpha} \ell^2 (\partial_t \widetilde{y}_\Gamma)^2 \,\d S\,\d t \right]. \label{e7s2}
\end{align}
Combining this estimate with \eqref{e5s2} and Lemma \ref{lems1}, we deduce
\begin{align*}
& \int_\Omega \mathrm{e}^{-2s\alpha(\theta, x)} v^2(\theta, x) \,\d x +  \int_\Gamma \mathrm{e}^{-2s\alpha(\theta, x)} v^2_\Gamma(\theta, x) \,\d S  \leq C \left[s^{3/2} \lambda^{5/2} \int_{\omega_{t_0, t_1}}\mathrm{e}^{-2s\alpha}\xi^3 v^2 \,\d x\,\d t \right. \nonumber\\
& \hspace{1cm} \left. +\, s^{-3/2} \lambda^{-3/2} \int_{\Omega_{t_0, t_1}}\mathrm{e}^{-2s\alpha} a^2 (\partial_t \widetilde{y})^2 \,\d x\,\d t + s^{-3/2} \lambda^{-3/2} \int_{\Gamma_{t_0, t_1}}\mathrm{e}^{-2s\alpha} \ell^2 (\partial_t \widetilde{y}_\Gamma)^2 \,\d S\,\d t \right]. 
\end{align*}
This with \eqref{e4s2} yield the result.
\end{proof}

\begin{proof}[Proof of Theorem \ref{thm1}]
By inequality \eqref{bd1} and the fact that $\alpha$ attains its minimum at $\theta$, we have
\begin{align*}
& -C s^{-3/2} \lambda^{-3/2} \int_{\Omega} \mathrm{e}^{-2s\alpha(\theta,x)} a^2(x) \,\d x -C s^{-3/2} \lambda^{-3/2} \int_{\Gamma}\mathrm{e}^{-2s\alpha(\theta,x)} \ell^2(x) \,\d S\\
& \qquad \leq -C s^{-3/2} \lambda^{-3/2} \int_{\Omega_{t_0, t_1}}\mathrm{e}^{-2s\alpha} a^2 (\partial_t \widetilde{y})^2 \,\d x\,\d t -C s^{-3/2} \lambda^{-3/2} \int_{\Gamma_{t_0, t_1}}\mathrm{e}^{-2s\alpha} \ell^2 (\partial_t \widetilde{y}_\Gamma)^2 \,\d S\,\d t.
\end{align*}
Then, Lemma \ref{lems2} implies that
\begin{align}
&\int_\Omega \mathrm{e}^{-2s\alpha(\theta, x)} a^2([\widetilde{y}(\theta, x)]^2 -C s^{-3/2} \lambda^{-3/2}) \,\d x + \int_\Gamma \mathrm{e}^{-2s\alpha(\theta, x)} \ell^2([\widetilde{y}_\Gamma(\theta, x)]^2 -C s^{-3/2} \lambda^{-3/2}) \,\d S \nonumber \\
& \qquad \leq C \left[s^{3/2} \lambda^{5/2} \int_{\omega_{t_0, t_1}}\mathrm{e}^{-2s\alpha}\xi^3 v^2 \,\d x\,\d t + \|Z(\theta, \cdot)\|_{\mathbb{H}^2}^2\right]. \label{e1s3}
\end{align}
Using Lemma \ref{lempos1}, we obtain
$$\widetilde{y}^2(\theta,x) \geq \mathrm{e}^{-2RT} r^2 >0 \quad\text{ and } \quad\widetilde{y}_\Gamma^2(\theta,x) \geq \mathrm{e}^{-2RT} r^2 >0.$$
It follows that
\begin{align*}
\widetilde{y}^2(\theta,x) -C s^{-3/2} \lambda^{-3/2} &\geq \mathrm{e}^{-2RT} r^2 -C s^{-3/2} \lambda^{-3/2}\ge r_1 >0,\\
\widetilde{y}_\Gamma^2(\theta,x) -C s^{-3/2} \lambda^{-3/2} &\geq r_1 >0
\end{align*}
for some $r_1$ independent on $s$ and $\lambda$ sufficiently large. Hence, the inequality \eqref{e1s3} becomes
\begin{align}
&\int_\Omega \mathrm{e}^{-2s\alpha(\theta, x)} a^2(x) \,\d x + \int_\Gamma \mathrm{e}^{-2s\alpha(\theta, x)} \ell^2(x) \,\d S \leq \frac{C}{r_1} \left(s^{3/2} \lambda^{5/2} \int_{\omega_{t_0, t_1}}\mathrm{e}^{-2s\alpha}\xi^3 v^2 \,\d x\,\d t + \|Z(\theta, \cdot)\|_{\mathbb{H}^2}^2\right). \label{e2s3}
\end{align}
Since $\inf\limits_{x\in \Omega} \mathrm{e}^{-2s\alpha(\theta, x)} >0$ and $\sup\limits_{(t,x)\in \overline{\Omega}_{t_0, t_1}} \mathrm{e}^{-2s\alpha(t,x)} \xi^3(t,x) < +\infty $, by replacing $Z=Y-\widetilde{Y}$, $a=p-\widetilde{p}$ and $\ell=q-\widetilde{q}$, and fixing $\lambda$ and $s$ sufficiently large, we obtain
$$\|p-\widetilde{p}\|_{L^2(\Omega)}^2 + \|q-\widetilde{q}\|_{L^2(\Gamma)}^2 \leq C \left(\|(Y-\widetilde{Y})(\theta, \cdot)\|_{\mathbb{H}^2}^2 + \|\partial_t (y-\widetilde{y})\|_{L^2\left(\omega_{t_0,t_1}\right)}^2\right).$$
Thus, the proof of Theorem \ref{thm1} is complete.
\end{proof}

As a direct corollary of Theorem \ref{thm1}, we have the following uniqueness result.
\begin{corollary}
Under the same assumptions of Theorem \ref{thm1}, if
\begin{align*}
y(\theta, \cdot)&=\widetilde{y}(\theta, \cdot) \quad\quad\text{ in }\Omega,\\
\partial_t y(t,x) &=\partial_t\widetilde{y}(t,x) \quad\text{ in } \omega_{t_0,t_1},
\end{align*}
then $p=\widetilde{p}$ in $\Omega$ and $q= \widetilde{q}$ in $\Gamma$.
\end{corollary}

\subsection{Proof of the stability estimate for initial data}\label{subsec3.3}
Using the logarithmic convexity method (see Appendix \ref{ap1}), we prove Proposition \ref{pthm1} which is based on Theorem \ref{thm1}.
\begin{proof}[Proof of Proposition \ref{pthm1}]
Throughout the proof, $C$ will denote a generic constant which is independent of initial data. Recall that $Z=Y-\widetilde{Y}$, $Z_0=Z(0)$, $a=p-\widetilde{p}$ and $\ell=q-\widetilde{q}$. Then, $V=\partial_t Z$ is the solution of
\begin{empheq}[left = \empheqlbrace]{alignat=2}
\begin{aligned}
&\partial_t v =\dv(A(x) \nabla v) + B(x)\cdot \nabla v +p(x) v+a(x)\partial_t\widetilde{y}, &\text{ in } \Omega_T ,\\
&\partial_t v_{\Gamma}=\dv_{\Gamma} (D(x)\nabla_\Gamma v_{\Gamma}) -\partial_{\nu}^A v + \langle b(x), \nabla_\Gamma v_{\Gamma}\rangle_{\Gamma}+q(x) z_{\Gamma}+\ell(x)\partial_t\widetilde{y}_{\Gamma}, &\text{on } \Gamma_T,\\
& v_{\Gamma}(t,x)= v_{|\Gamma}(t,x), &\text{on } \Gamma_T ,\\
& v(0,x)=\dv(A(x)\nabla z_0)+ B(x)\cdot \nabla z_0+p z_0 +a \widetilde{y}_0, &\text{in } \Omega,\\
& v_{\Gamma}(0,x) =\dv_{\Gamma} (D(x)\nabla_\Gamma z_{0,\Gamma})-\partial_{\nu}^A z_0 + \langle b,\nabla_\Gamma z_{0,\Gamma} \rangle_\Gamma +q z_{0,\Gamma}+\ell \widetilde{y}_{0,\Gamma}, &\text{on } \Gamma.
\label{pe1tope3}
\end{aligned}
\end{empheq}
Consider $W=(w,w_\Gamma)$ the solution of
\begin{empheq}[left = \empheqlbrace]{alignat=2}
\begin{aligned}
&\partial_t w =\dv(A(x) \nabla w) + B(x)\cdot \nabla w +p w+a(x)\partial_t\widetilde{y}, &\text{ in } \Omega_T ,\\
&\partial_t w_{\Gamma}=\dv_{\Gamma} (D(x)\nabla_\Gamma w_{\Gamma}) -\partial_{\nu}^A w + \langle b(x), \nabla_\Gamma w_{\Gamma}\rangle_{\Gamma}+q w_{\Gamma}+\ell(x)\partial_t\widetilde{y}_{\Gamma}, &\text{on } \Gamma_T,\\
&w_{\Gamma}(t,x)= w_{|\Gamma}(t,x), &\text{ on } \Gamma_T ,\\
&W(0)=(0,0) &\Omega\times\Gamma . \label{pe6tope9}
\end{aligned}
\end{empheq}
The function $U=V-W$ satisfies
\begin{empheq}[left = \empheqlbrace]{alignat=2}
\begin{aligned}
&\partial_t u =\dv(A(x) \nabla u) + B(x)\cdot \nabla u +p(x) u, &\hspace{1.5cm} \text{ in } \Omega_T ,\\
&\partial_t u_{\Gamma}=\dv_{\Gamma} (D(x)\nabla_\Gamma u_{\Gamma}) -\partial_{\nu}^A u + \langle b(x), \nabla_\Gamma u_{\Gamma}\rangle_{\Gamma}+q(x) u_{\Gamma}, &\text{on } \Gamma_T,\\
&u_{\Gamma}(t,x)= u_{|\Gamma}(t,x), &\text{ on } \Gamma_T ,\\
&U(0)=V(0), &\Omega\times\Gamma .\label{pe10tope11}
\end{aligned}
\end{empheq}
By Duhamel's formula $\displaystyle W(t)=\int_0^t \mathcal{T}(t-\tau)[a\,\partial_\tau \widetilde{y}, \ell\,\partial_\tau \widetilde{y}_\Gamma] \,\d \tau$, we have
\begin{align*}
\|W(t)\|_{\mathbb{L}^2} &\leq  C \int_0^T \|(a\,\partial_\tau \widetilde{y},\ell\,\partial_\tau \widetilde{y}_\Gamma)\|_{\mathbb{L}^2} \,\d \tau, \quad 0\leq t \leq T.
\end{align*}
On the other hand, using inequality \eqref{bd2}, we derive
$$\|(a\,\partial_\tau \widetilde{y},\ell\,\partial_\tau \widetilde{y}_\Gamma)\|_{\mathbb{L}^2}^2=\int_\Omega a^2(x) (\partial_\tau \widetilde{y})^2 \,\d x + \int_\Gamma \ell^2(x) (\partial_\tau \widetilde{y}_\Gamma)^2 \,\d S \leq C \|(a,\ell)\|_{\mathbb{L}^2}^2.$$
Therefore,
\begin{equation}
\|W(t)\|_{\mathbb{L}^2} \leq C \|(a,\ell)\|_{\mathbb{L}^2} , \quad 0\leq t \leq T. \label{inq1}
\end{equation}
Using $Y_0, \widetilde{Y}_0 \in \mathcal{I}$, we obtain
\begin{align*}
\|U(0)\|_{\mathbb{L}^\infty} =\|V(0)\|_{\mathbb{L}^\infty}&=\|Z_t(0)\|_{\mathbb{L}^\infty} \leq M \qquad \text{and} \qquad \|U(0)\|_{\mathbb{L}^2} \le M_1
\end{align*}
for some positive constant $M_1$. By Proposition \ref{prop21}, the operator $\mathcal{A}$ given by \eqref{op1} generates an analytic semigroup of angle $\dfrac{\pi}{2}$ on $\mathbb{L}^2$. Hence, applying Lemma \ref{lemA4} to \eqref{pe10tope11}, we obtain
\begin{equation}
\|U(t)\|_{\mathbb{L}^2} \leq K M_1^{1-\frac{t}{\theta}} \|U(\theta)\|_{\mathbb{L}^2}^{\frac{t}{\theta}} , \quad 0\leq t\leq \theta. \label{logeq1}
\end{equation}
Since $Z_t=U+W$, by \eqref{inq1} and \eqref{logeq1}, we have
\begin{align}
\|Z_t(t)\|_{\mathbb{L}^2} & \leq \|U(t)\|_{\mathbb{L}^2} + \|W(t)\|_{\mathbb{L}^2} \nonumber\\
& \leq C \|U(\theta)\|_{\mathbb{L}^2}^{\frac{t}{\theta}} + C \|(a,\ell)\|_{\mathbb{L}^2} \nonumber\\
& \leq C (\|Z_t(\theta)\|_{\mathbb{L}^2}+ \|(a,\ell)\|_{\mathbb{L}^2})^{\frac{t}{\theta}} + C \|(a,\ell)\|_{\mathbb{L}^2} \label{inq2}
\end{align}  
for $0 \leq t \leq \theta$. Using \eqref{bd2} with trace theorem and
\begin{align*}
\partial_t z(\theta,x)&=\dv(A(x) \nabla z(\theta,x))+ B(x)\cdot \nabla z(\theta,x)+p(x)z(\theta,x)+a(x)\widetilde{y}(\theta,x),\\
\partial_t z_\Gamma(\theta,x)&=\dv_{\Gamma} (D(x)\nabla_\Gamma z_{\Gamma}(\theta,x))-\partial_{\nu}^A z(\theta,x)+ \langle b(x), \nabla_\Gamma z_{\Gamma}(\theta,x)\rangle_{\Gamma} +q(x)z_\Gamma(\theta,x) +\ell(x)\widetilde{y}_\Gamma(\theta,x),
\end{align*}
we infer that
\begin{align*}
\|\partial_t z(\theta,\cdot)\|_{L^2(\Omega)}^2 &\leq C \|z(\theta,\cdot)\|_{H^2(\Omega)}^2 + C\|a\|_{L^2(\Omega)}^2,\\
\|\partial_t z_\Gamma(\theta,\cdot)\|_{L^2(\Gamma)}^2 &\leq C \|z_\Gamma(\theta,\cdot)\|_{H^2(\Gamma)}^2 + C \|z(\theta,\cdot)\|_{H^2(\Omega)}^2 + C\|\ell\|_{L^2(\Gamma)}^2.
\end{align*}
Hence
$$\|Z_t(\theta)\|_{\mathbb{L}^2} \leq C \|Z(\theta)\|_{\mathbb{H}^2}+ C \|(a,\ell)\|_{\mathbb{L}^2}.$$
Then, the inequality \eqref{inq2} yields
$$\|Z_t(t)\|_{\mathbb{L}^2} \leq C (\|Z(\theta)\|_{\mathbb{H}^2}+ \|(a,\ell)\|_{\mathbb{L}^2})^{\frac{t}{\theta}} + C \|(a,\ell)\|_{\mathbb{L}^2}, \qquad 0 \leq t \leq \theta.$$
Consequently, we obtain
\begin{align*}
\|Y_0 -\widetilde{Y}_0\|_{\mathbb{L}^2} &= \|Z(0)\|_{\mathbb{L}^2} =\left\| -\int_0^{\theta} Z_\tau(\tau) \,\d \tau + Z(\theta)\right\|_{\mathbb{L}^2}\\
&\leq C \int_0^{\theta} (\|Z(\theta)\|_{\mathbb{H}^2}+ \|(a,\ell)\|_{\mathbb{L}^2})^{\frac{\tau}{\theta}} \,\d \tau + C\theta \|(a,\ell)\|_{\mathbb{L}^2} + \|Z(\theta)\|_{\mathbb{L}^2}\\
&= C\theta \frac{(\|Z(\theta)\|_{\mathbb{H}^2}+ \|(a,\ell)\|_{\mathbb{L}^2})-1}{\log(\|Z(\theta)\|_{\mathbb{H}^2}+ \|(a,\ell)\|_{\mathbb{L}^2})} + C\theta \|(a,\ell)\|_{\mathbb{L}^2} + \|Z(\theta)\|_{\mathbb{L}^2}\\
& \le C \left(\frac{E-1}{\log E} + E \right),
\end{align*}
where we denoted $E:=\|Z(\theta)\|_{\mathbb{H}^2}+ \|(a,\ell)\|_{\mathbb{L}^2}$. By Theorem \ref{thm1}, we have
$$\|(a,\ell)\|_{\mathbb{L}^2} \le C\left(\|Z(\theta)\|_{\mathbb{H}^2} + \|\partial_t z \|_{L^2\left(\omega_{t_0, t_1}\right)}\right).$$
Therefore, when $\|Z(\theta)\|_{\mathbb{H}^2} + \|\partial_t z \|_{L^2\left(\omega_{t_0, t_1}\right)}$ is sufficiently small, we obtain
\begin{equation}
0 < E \le C_1\left(\|Z(\theta)\|_{\mathbb{H}^2} + \|\partial_t z \|_{L^2\left(\omega_{t_0, t_1}\right)}\right) <1 \label{eqt1}
\end{equation}
for some constant $C_1>0$. Using the inequality $\dfrac{\tau -1}{\log \tau} +\tau \leq -\dfrac{1+ \mathrm{e}^{-2}}{\log \tau}$ for $0<\tau <1$ and \eqref{eqt1}, we deduce that
$$\|Y_0 -\widetilde{Y}_0\|_{\mathbb{L}^2} \le \frac{-C}{\log \left(C_1\left(\|Z(\theta)\|_{\mathbb{H}^2} + \|\partial_t z \|_{L^2\left(\omega_{t_0, t_1}\right)}\right)\right)}.$$
Finally, the result follows by writing $Z=Y-\widetilde{Y}$.
\end{proof}

\begin{remark}
Our method differs from the one in \cite{CGR'06}, which uses Theorem \ref{thmc2}. The latter requires more regularity on coefficients $B$ and $b$ to write $\mathcal{A}=\mathcal{A}_+ +\mathcal{A}_-$, where
\begin{align*}
\begin{aligned}
&\mathcal{A}_+=\begin{pmatrix} \dv(A\nabla\cdot)-\dfrac{1}{2} \dv B + p & 0\\ -\partial_\nu^A & \dv_\Gamma(D\nabla_\Gamma\cdot) +\dfrac{1}{2}(B\cdot \nu -\dv_\Gamma b) + q \end{pmatrix},\\
&\mathcal{A}_-=\begin{pmatrix} \dv(\cdot B)-\dfrac{1}{2} \dv B & 0\\ 0 & \dv_\Gamma(\cdot b) -\dfrac{1}{2}(\dv_\Gamma b+ B\cdot \nu)\end{pmatrix}.
\end{aligned}
\end{align*}
It also requires more regular initial states. See \cite[Theorem 4.3]{CGR'06}, where the authors dealt with initial values in $H^4(\Omega)$.
\end{remark}

\appendix
\section{Logarithmic convexity}\label{ap1}
For completeness, we summarize some important results on the logarithmic convexity method used in the proof of Proposition \ref{pthm1}.

Let $(H, \langle \cdot, \cdot \rangle)$ be a real Hilbert space and $A \colon D(A)\subset H \longrightarrow H$ a linear operator. Consider the following abstract Cauchy problem
\begin{empheq}[left = \empheqlbrace]{alignat=2}
\begin{aligned}
& u'(t)=A u(t), \quad t \geq 0,\\
& u(0)=u_0 \in H. \label{A1}
\end{aligned}
\end{empheq}
The basic idea of logarithmic convexity for \eqref{A1} is that the solution is small at any intermediate time, provided that it is small at an arbitrary fixed time, given a bounded initial data.

\noindent\textbf{Case 1: $A$ is self-adjoint and generates a $C_0$-semigroup.}\\
The logarithmic convexity method in this case is due to Agmon and Nirenberg \cite{AN'63}. The following result can be found in \cite{GaT'11}.
\begin{lemma}
Consider $A$ a self-adjoint operator generator of a $C_0$-semigroup on $H$. Let $\theta >0$ and $M>0$ be fixed. For all $u_0\in H$ such that $\|u_0\|\leq M$, the solution of \eqref{A1} satisfies
\begin{equation}
\|u(t)\| \leq M^{1- \frac{t}{\theta}} \|u(\theta)\|^{\frac{t}{\theta}} \label{eqlc1}
\end{equation}
for all $t\in [0,\theta]$.
\end{lemma}

\begin{remark}
A self-adjoint operator $A$ on $H$ generates a $C_0$-semigroup if and only if it is bounded above, i.e., there exists $\kappa \in \mathbb{R}$ such that
$$\langle Ax,x\rangle \leq \kappa \|x\|^2 \qquad \text{ for all } x\in D(A),$$
see, e.g., \cite[Proposition 3.28, pp. 91]{EN'99}. In this case, the semigroup $\left(\mathrm{e}^{tA}\right)_{t\ge 0}$ generated by $A$ is analytic of angle $\dfrac{\pi}{2}$, see, for instance, \cite[Corollary 4.7]{EN'99}.
\end{remark}

\noindent\textbf{Case 2: $A$ is subordinated to its symmetric part.}\\
Next, we state an extension of the previous logarithmic convexity that can be found in \cite[Theorem 3.1.3]{Is'17}. Consider the following inequality
\begin{equation}
\|\partial_t u + Au\| \leq \gamma\|u\|, \qquad \text{ on } (0,\theta) \label{c2eq1}
\end{equation}
for some positive constant $\gamma$. We assume that $A$ takes the form $A=A_+ + A_-$, where $A_+$ is a symmetric operator with domain $D(A)$ and $A_-$ is skew-symmetric satisfying the following conditions
\begin{align}
\|A_- u\|^2 &\leq \gamma\left(\|A_+ u\| \|u\|+ \|u\|^2\right), \label{c2eq2}\\
\partial_t \langle A_+ u, u\rangle &\leq 2\langle A_+ u, \partial_t u\rangle + \gamma\left(\|A_+ u\|\|u\|+\|u\|^2\right). \label{c2eq3}
\end{align}

\begin{theorem}\label{thmc2}
Let $M>0$ and $u\in C^1([0,T];H)$ be a solution of inequality \eqref{c2eq1} such that $u(t)\in D(A)$ for all $t\in [0,T]$, and $\|u(0)\|\leq M$. Then,
\begin{equation}
\|u(t)\| \leq C_1 M^{1- \mu(t)} \|u(\theta)\|^{\mu(t)}, \qquad 0\leq t\leq \theta \label{c2eq4}
\end{equation}
for some constant $C_1=C\left(\gamma,\theta\right)$ and $\mu(t)=\dfrac{1-\mathrm{e}^{-Ct}}{1-\mathrm{e}^{-C\theta}}$, where $C$ is a constant depending on $\gamma$. Moreover, if $\gamma=0$ we can choose $C_1 =1$ and $\mu(t)=\dfrac{t}{\theta}$ for all $t\in [0,\theta]$.
\end{theorem}

\noindent\textbf{Case 3: $A$ generates an analytic $C_0$-semigroup of angle $\dfrac{\pi}{2}$.}\\
A generalization of logarithmic convexity inequality \eqref{eqlc1} to the class of operators generating analytic semigroups of angle $\psi \in \left(0,\dfrac{\pi}{2}\right]$ (even in Banach spaces) was established in \cite{Mi'75}.

Here we state the result in a restricted form which is the best suited
for our case.
\begin{lemma}\label{lemA4}
Let $\theta >0$, $M>0$ be fixed and $u_0\in H$ such that $\|u_0\|\leq M$. If $A$ generates an analytic $C_0$-semigroup of angle $\dfrac{\pi}{2}$, then the solution of \eqref{A1} satisfies
\begin{equation}
\|u(t)\| \leq K M^{1- \frac{t}{\theta}} \|u(\theta)\|^{\frac{t}{\theta}}
\end{equation}
for all $t\in [0,\theta]$, where $K\ge 1$ is a constant depending on the semigroup.
\end{lemma}

\end{document}